\documentclass[12pt,leqno]{article}
\usepackage{amssymb, graphicx, amscd, amsmath, amsthm, 
array, fancyvrb, relsize, paralist,url}
\usepackage{listings}
\usepackage{epstopdf}
\providecommand{\U}[1]{\protect\rule{.1in}{.1in}}
\makeatletter
\g@addto@macro\bfseries{\boldmath}
\makeatother
\newtheorem{thm}{Theorem}[section]
\newtheorem{lem}[thm]{Lemma}

\def\A{\mathcal{A}}

\def\S{\mathcal{S}}
\def\T{\mathcal{T}}
\def\L{\mathcal{L}}

\numberwithin{equation}{section}

\title{Nullities for a class of skew-symmetric Toeplitz band matrices}

\author{\\ Ron Evans\\
Department of Mathematics\\
University of California at San Diego\\
La Jolla, CA  92093-0112 \\
revans@ucsd.edu
\\ \\
John Greene\\
Department of Mathematics and Statistics\\
University of Minnesota--Duluth\\
Duluth, MN  55812\\
jgreene@d.umn.edu
\\ \\
Mark Van Veen\\
2138 Edinburg Avenue\\
Cardiff by the Sea, CA 92007  \\
mavanveen@ucsd.edu
}

\date{February 2020}

\begin{document}

\maketitle



\bigskip

\begin{abstract}
\noindent
For all
$n > k \ge 1$, we give formulas for the nullity $N(n,k)$  of the
$n \times n$ skew-symmetric  Toeplitz band matrix
whose first $k$ superdiagonals have all entries $1$ and whose
remaining superdiagonals have all entries $0$.
This is accomplished by counting the number
of cycles in certain directed graphs.
As an application, for each fixed integer $z\ge 0$ and large fixed $k$,
we give an asymptotic formula for the percentage of $n > k$
satisfying $N(n,k)=z$.  For the purpose of rapid computation,
an algorithm is devised that quickly computes $N(n,k)$ even for extremely
large values of $n$ and $k$.

\end{abstract}
\maketitle

\section{Introduction}

For $n > k \ge 1$ and $x \in \mathbb{R}$ , 
let $A(n,k,x)$ denote the $n \times n$ skew-symmetric
Toeplitz matrix whose first $k$ superdiagonals have all entries $1$,
and whose remaining superdiagonal entries are all $-x$.  
For example,  $A(6,2,x)$ is the matrix
\[
\left[
\begin{array}
[c]{cccccc}
0 & 1 & 1 & -x & -x & -x\\
-1 & 0 & 1 & 1 & -x & -x\\
-1 & -1 & 0 & 1 & 1 & -x\\
x & -1 & -1 & 0 & 1 & 1\\
x & x  & -1 & -1 & 0 & 1\\
x & x & x & -1 & -1 & 0
\end{array}
\right]  .
\]

The matrices $A(n,k,x)$ are payoff matrices for the integer choice
matrix games discussed for example in \cite{EH,EW,EW2}.
Let $N(n,k)$ denote the nullity of the
skew-symmetric Toeplitz
band matrix $A(n,k):=A(n,k,0)$.  
The primary goal of this paper is to determine the nullities
$N(n,k)$  for all $n > k \ge 1$.  
Our methods are different from those of Price et al.
\cite{Price1,Price2}, who investigated nullity sequences of
Toeplitz matrices over finite fields. We say more about the benefit
of our approach in the second paragraph of Section 6.

There is an intriguing connection between the nullity $N(n,k)$
and the payoff matrix $A(n,k,x)$ when $n$ is even.
Let $D(n,k,x) \in \mathbb{Z}[x]$
denote the determinant of $A(n,k,x)$.
When $n$ is even,
we conjecture that the zero
$x=0$ of $D(n,k,x)$ has multiplicity $N(n,k)$.
For example, the term of smallest degree in $D(16,4,x)$ is $81x^4$, 
and $N(16,4)=4$. 
(If $n$ is odd, then of course the polynomial $D(n,k,x)$ vanishes
identically, since the transpose of $A(n,k,x)$ equals $-A(n,k,x)$.)

In Section 2, for each pair of integers $n,k$ with $k \ge 1$, 
we construct a directed graph $G(n,k)$ 
on the vertices $0, 1, \dots, k$ which consists of a disjoint union
of pure cycles and an open path called a ``tail".  For example, 
$G(16,8)$ consists of the two cycles
\[
7 \rightarrow 1 \rightarrow 4 \rightarrow 7, \quad
6 \rightarrow 0 \rightarrow 3 \rightarrow 6,
\]
together with the tail $\ 8 \rightarrow 2 \rightarrow 5$.
Theorem 2.1 shows that $N(n,k)$ equals the number of cycles in $G(n,k)$
(so for example $N(16,8)=2$).
As shown at the end of Section 2, Theorem 2.1 enables one to rapidly
compute the nullity $N(n,k)$ using Mathematica, even for extremely
large $n$, $k$.

In order to obtain formulas for $N(n,k)$, we undertake a detailed
analysis of the structure of the graph $G(n,k)$ in Sections 3--6,
for each fixed $k$.
Theorem 3.9, the main result in Section 3, proves 
that for any fixed pair $k$, $n$, the cycles
in $G(n,k)$ are ``translates" of each other.
In particular, these cycles all have the same length.  
Section 4 focuses on properties of the tail of $G(n,k)$.
In Section 5, we prove lemmas that show how the number of cycles in $G(n,k)$
changes as $n$ increases.
These results are then applied in Section 6 
to describe the shape
of the line graph connecting the points 
\[
(n, N(n,k)), \quad 0 \le n \le k^2+k.
\]

Theorem 7.5, our main result, determines each location $n$  
for which the nullity $N(n,k)$ has a local maximum in our line graph.
We can then exploit the shape of the line graph to determine the values of all
nullities $N(n,k)$, as is described  at the beginning of Section 7.

Applications and examples are given in Section 8.
For instance, Theorems 8.1 and 8.2 show that for large fixed $k$, 
about $30.4\%$ of the matrices $A(n,k)$ are nonsingular,
about $38\%$ have nullity $1$, and about $11\%$ have nullity $2$.
The percentages continue to decrease as the nullities increase further.

In Theorem 9.2, we prove the aforementioned conjecture
on the determinant $D(n,k,x)$ for all
$k \ge (n-2)/2$.   The conjecture remains open 
in general for $k \le (n-4)/2$.

\section{The directed graph $G(n,k)$} 
Let $\zeta$ be a complex primitive $k$-th root of unity and let 
$\xi$ be a complex primitive
$(k+1)$-th root of unity.  For integers $\ell$, define the 
$1 \times k$ row vectors
\[
u(\ell)=(1, \zeta^{\ell}, \zeta^{2 \ell},\dots, \zeta^{(k-1) \ell}),\quad
v(\ell)=(\xi^{\ell},\xi^{2 \ell}, \dots, \xi^{k \ell}).
\]
Throughout this paper, when $k$ or $k+1$ is explicitly employed as a subscript 
on an integer $\ell$, 
define $\ell_{k}$ (resp. $\ell_{k+1}$) to be the least nonnegative residue
of $\ell$ modulo $k$ (resp. $k+1$). Note that $u(\ell)=u(\ell_{k})$
and $v(\ell) = v(\ell_{k+1})$.

Define the $2k \times 2k$ matrix
\begin{equation}  \label{eq:2.1}
V(n, k) = \left( \begin{tabular}{c|c} $u(0)$ & $v(0)$ \\ 
\vdots & \vdots \\ $u(k-1)$ & $v(k-1)$ \\ 
\hline $u(k+n)$ & $v(k+n)$ \\ \vdots & \vdots \\ 
$u(2k-1 + n)$ & $v(2k-1 + n)$ \end{tabular} \right) = 
\left( \begin{tabular}{c|c} $A$ & $B$ 
\\ \hline $C$ & $D$ \end{tabular} \right),
\end{equation}
where $A, B, C, D$ are $k \times k$ matrices.
Note that $V(0,k)$ is a Vandermonde matrix with
distinct columns, so that its rows are independent over $\mathbb{C}$.

Converting the notation in \cite{T} to ours, we see that our matrix
$V(n,k)$ is the matrix in \cite[eq. (6)]{T}.  Thus by \cite[eq. (14)]{T},
$N(n,k)$ equals the nullity of $V(n,k)$.
Consequently, for
fixed $k$, $N(n,k)$ depends only on the value of 
$n \pmod{k^2+k}$.  We now drop the condition $n > k$ and extend the 
definition of $N(n,k)$ so that it equals the nullity of $V(n,k)$
for {\em all} integers $n$.   In particular, $N(0,k) = 0$.

The rows of $C$ in \eqref{eq:2.1} are a cyclic permutation of the rows of $A$.  
Thus after row reduction,
$V(n,k)$ can be converted to
$ \left( \begin{tabular}{c|c} $A$ & $B$ \\ 
\hline $0$ & $H$ \end{tabular} \right)$, where the
$i$-th row of the $k \times k$ matrix $H$ is 
\[
v((n+k + i - 1)_{k+1}) - v((n + k + i - 1)_{k}), \quad 1 \le i \le k.
\]
Since $H$ and $V(n,k)$ have the same nullity, $H$ has nullity $N(n,k)$.
In particular, $N(n,k) \le k$ for every $n$.

Define a directed graph $G(n,k)$  on the vertices $0, 1, \dots, k$ with 
edges $a \rightarrow b$ directed from $a$ to $b$ if and only if 
$v(a) - v(b)$ is a row in $H$.  If $a = b$ we consider this edge 
to be a loop, i.e., a cycle of length 1.  
In Theorem 2.1 below, we will show that the nullity of $H$
equals the number of cycles in the graph $G(n,k)$.

The graph $G(n,k)$ has exactly $k$ edges, namely
\begin{equation}\label{eq:2.2}
E(i) = E(i,n):=
(i + n -2)_{k+1} \rightarrow  (i + n - 1)_{k}, \quad 1 \le i \le k.
\end{equation}
Since no endpoint of an edge can be $k$, the vertex $k$ has in-degree $0$.
Since no initial point of an edge can be $(n-2)_{k+1}$, the vertex
$(n-2)_{k+1}$ has out-degree $0$.   For every other vertex, the in-degree
and out-degree are both $1$.  Thus $G(n,k)$ consists of an open path (called
the ``tail") connecting initial vertex $k$ to terminal vertex $(n-2)_{k+1}$,
together with a (possibly empty) disjoint union of simple cycles.
The number of edges in a path $P$ is the path length, denoted by $|P|$. 
If it happens that
$k$ = $(n-2)_{k+1}$,  then the tail has length $0$ and $k$ is an isolated
vertex.

For a cycle $C$ in $G(n,k)$, let $R(C)$ denote the set of row positions
in the matrix $H$ corresponding to the edges of $C$.
For example, if $C$ is the cycle
\[
1 \rightarrow 4 \rightarrow 8 \rightarrow 5 \rightarrow 9 \rightarrow 1,
\]
then $R(C)$ is the set of row positions of
$v(1) - v(4)$, $v(4) - v(8)$, $v(8) - v(5)$, $v(5) - v(9)$, $v(9) - v(1)$  
in the matrix $H$. 
Observe that $C$ corresponds to a $1 \times k$ 
row vector of $0$'s and $1$'s
in the left nullspace of $H$, where the $1$'s are in the positions 
matching the row positions in $R(C)$.

\begin{thm}
The nullity $N(n,k)$ of the matrix $H$ equals the number of cycles
in the graph $G(n,k)$.
\end{thm}

\begin{proof}
As noted above,
each cycle corresponds to a row vector in the left nullspace
$\L$ of $H$.  Since the cycles are disjoint, the corresponding row vectors
are independent over $\mathbb{C}$.   
These row vectors therefore span a subspace
$\A \subset \L$ whose dimension equals the number of cycles.  
It remains to show that $\A = \L$.

Let $\T$ denote the set of all (nonzero) complex row vectors
$\tau  \in \L$ such that  $\tau \notin \A$.
Our goal is to show that $\T$ is empty.

Suppose for the purpose of contradiction that $\T$ contains
a vector $\tau$. The product $\tau H$ is a linear combination
of the rows of $H$ whose coefficients are the entries in $\tau$.  
This linear combination
sums to 0, since $\tau \in \L$.  

For each cycle $C$, consider the $|C|$ entries in $\tau$ 
situated in the positions
matching the row positions in $R(C)$.
Replace one of these $|C|$
coefficients (call it $c$) by 0 and then subtract $c$ from
each of the remaining $|C|-1$  coefficients.
Since such replacements are made for every cycle $C$,
this yields a new vector $\tau' \in \T$ with the property that
$\tau' H$ is a linear combination of rows of $H$
corresponding to edges of disjoint {\em open} paths.
For example, the subsum of $\tau' H$ corresponding to an
open path such as
\[
2 \rightarrow 8 \rightarrow 6 \rightarrow 3 \rightarrow 7
\]
would have the form
\begin{equation}\label{eq:2.3}
d_1(v(2) - v(8)) + d_2(v(8) - v(6)) + d_3(v(6) - v(3))+
d_4(v(3) - v(7)).
\end{equation}

The only dependence relations among the vectors $v(0), \dots , v(k)$
are
\[
\alpha (v(0) + v(1) + \dots + v(k)) = 0, \quad \alpha \in \mathbb{C}.
\]
Since $\tau' H = 0$, 
the coefficients of the vectors $v(\ell)$ in the expansion of $\tau' H$
must all be equal.   

Let us first focus on the subsum in our example \eqref{eq:2.3}.
Expanding in $\tau' H$, this subsum becomes
\[
d_1 v(2) +(d_2-d_1)v(8)+(d_3-d_2)v(6)+(d_4-d_3)v(3) -d_4 v(7).
\]
Thus
\[
d_1 = d_2-d_1 = d_3-d_2 = d_4-d_3 = -d_4.
\]
The five equal members of this equality sum to zero, so all five must vanish.
Therefore  the $d_i$ all vanish.
This type of argument applies generally to each of the disjoint subsums of
$\tau' H$ corresponding to an open path,
showing that all entries of $\tau'$ vanish.   This
contradicts the fact that $\tau' \in \T$.
\end{proof}

Due to Theorem 2.1, the
nullity of $A(n,k)$ can be rapidly computed 
for very large values of the arguments $n,k$, 
using this function in Mathematica:

\begin{verbatim}
FastNullity[n_,k_]:= 
Length[FindCycle[Table[Mod[i-2+n,1+k]->Mod[i-1+n,k],{i,1,k}]
,k, All]]+Sum[If[Mod[i-2+n,1+k]==Mod[i-1+n,k],1,0],{i,1,k}]
\end{verbatim}

\section{Properties of the cycles in $G(n,k)$}

Fix a pair $n$, $k$.
The object of this section is to prove Theorem 3.9,
which shows that the cycles of $G(n,k)$ are ``translates" of each other.
A consequence of Theorem 3.9  is the nontrivial fact that the cycles have the 
same length.   We begin with a string of eight lemmas.

Recall that 
when $k$ or $k+1$ is explicitly employed as a subscript 
on some integer $\beta$,
we have defined  $\beta_{k}$ (resp. $\beta_{k+1}$) 
to be the least nonnegative residue
of $\beta$ modulo $k$ (resp. $k+1$).

\begin{lem}
Suppose that
\begin{equation}\label{eq:3.1}
a_1 \rightarrow a_2 \rightarrow \dots \rightarrow a_{\ell}
\end{equation}
is a path in $G(n,k)$.
Then
\begin{equation}\label{eq:3.2}
(a_1+1)_{k+1} \rightarrow \dots \rightarrow (a_{\ell} +1)_{k+1}
\end{equation}
is a path in $G(n,k)$ if and only if
none of $a_2, \dots, a_{\ell}$
equals $k-1$ and none of
$a_1, \dots, a_{\ell -1}$ equals $(n-3)_{k+1}$.
\end{lem}

\begin{proof}
Suppose that \eqref{eq:3.2} holds.   Then none of $a_2, \dots, a_{\ell}$
can equal $k-1$, since $k$ has in-degree $0$, and 
none of
$a_1, \dots, a_{\ell -1}$ can equal $(n-3)_{k+1}$,
since $(n-2)_{k+1}$ has out-degree $0$.

Conversely, if 
none of
$a_1, \dots, a_{\ell -1}$ equals $(n-3)_{k+1}$,
then by \eqref{eq:2.2}, we have $i<k$ for any edge 
$E(i):=  a_\nu \rightarrow a_{\nu+1}$ 
in the path \eqref{eq:3.1}.
If moreover
none of $a_2, \dots, a_{\ell}$
equals $k-1$, then $E(i+1)$ is the edge
$(a_\nu +1)_{k+1} \rightarrow (a_{\nu +1} +1)_{k+1}$, so \eqref{eq:3.2} follows.
\end{proof}

\begin{lem}
Suppose that
\begin{equation}\label{eq:3.3}
a_1 \rightarrow  \dots \rightarrow a_{\ell -1} \rightarrow k-1
\end{equation}
is a path in $G(n,k)$.
Then
\begin{equation}\label{eq:3.4}
(a_1+1)_{k+1} \rightarrow \dots \rightarrow (a_{\ell -1} +1)_{k+1}
\rightarrow 0
\end{equation}
is a path in $G(n,k)$ if and only if
none of
$a_1, \dots, a_{\ell -1}$ equals $(n-3)_{k+1}$.
\end{lem}

\begin{proof}
If none of
$a_1, \dots, a_{\ell -1}$ equals $(n-3)_{k+1}$, then
$i < k$ for the edge
\[
E(i):= a_{\ell -1} \rightarrow k-1 =(i+n-1)_k,
\]
so that $E(i+1)$ is the edge
\[
(a_{\ell -1} +1)_{k+1} \rightarrow (i+n)_k = 0.
\]
\end{proof}

\begin{lem}
Suppose that
\begin{equation}\label{eq:3.5}
a_1 \rightarrow a_2 \rightarrow \dots \rightarrow a_{\ell}
\end{equation}
is a path in $G(n,k)$.
Then
\begin{equation}\label{eq:3.6}
(a_1-1)_{k+1} \rightarrow \dots \rightarrow (a_{\ell} -1)_{k+1}
\end{equation}
is a path in $G(n,k)$ if and only if
none of $a_2, \dots, a_{\ell}$
equals $0$ and none of
$a_1, \dots, a_{\ell -1}$ equals $(n-1)_{k+1}$.
\end{lem}

\begin{proof}
Suppose that \eqref{eq:3.6} holds.   Then none of $a_2, \dots, a_{\ell}$
can equal $0$, since $k$ has in-degree $0$, and
none of
$a_1, \dots, a_{\ell -1}$ can equal $(n-1)_{k+1}$,
since $(n-2)_{k+1}$ has out-degree $0$.

Conversely, if
none of
$a_1, \dots, a_{\ell -1}$ equals $(n-1)_{k+1}$,
then by \eqref{eq:2.2}, we have $i>1$ for any edge
$E(i):= a_\nu \rightarrow a_{\nu+1}$
in the path \eqref{eq:3.5}.
If moreover
none of $a_2, \dots, a_{\ell}$
equals $0$, then $E(i-1)$ is the edge
$(a_\nu -1)_{k+1} \rightarrow (a_{\nu +1} -1)_{k+1}$, so \eqref{eq:3.6} follows.
\end{proof}

\begin{lem}
Suppose that
\begin{equation}\label{eq:3.7}
a_1 \rightarrow  \dots \rightarrow a_{\ell -1} \rightarrow 0
\end{equation}
is a path in $G(n,k)$.
Then
\begin{equation}\label{eq:3.8}
(a_1-1)_{k+1} \rightarrow \dots \rightarrow (a_{\ell -1} -1)_{k+1}
\rightarrow k-1
\end{equation}
is a path in $G(n,k)$ if and only if
none of
$a_1, \dots, a_{\ell -1}$ equals $(n-1)_{k+1}$.
\end{lem}

\begin{proof}
If none of
$a_1, \dots, a_{\ell -1}$ equals $(n-1)_{k+1}$, then
$i >1$ for the edge
\[
E(i):= a_{\ell -1} \rightarrow 0 =(i+n-1)_k,
\]
so that $E(i-1)$ is the edge
\[
(a_{\ell -1} -1)_{k+1} \rightarrow (i+n-2)_k = k-1.
\]
\end{proof}

Let $C$ denote a cycle in $G(n,k)$ of the form
\begin{equation}\label{eq:3.9}
a_1 \rightarrow  \dots \rightarrow a_{\ell} \rightarrow a_1.
\end{equation}
If for some integer $t$,  $G(n,k)$ has a cycle
\[
(a_1 +t)_{k+1} \rightarrow \dots \rightarrow (a_{\ell} +t)_{k+1}
\rightarrow (a_1 +t)_{k+1},
\]
denote this cycle by $C+t$.  We refer to $C+t$ as a ``translate" of $C$.
\begin{lem}
For the cycle $C$ in \eqref{eq:3.9},
$C+1$ is a cycle  if and only if $C$ avoids the vertices
$k-1$ and $(n-3)_{k+1}$,
and $C-1$ is a cycle if and only if $C$ avoids the vertices
$0$ and $(n-1)_{k+1}$.
\end{lem}

\begin{proof}
This follows from Lemmas 3.1 and 3.3.
\end{proof}

\begin{lem}
For the cycle $C$ in \eqref{eq:3.9},
$C+t$ is a cycle for each $t \in [-v,u]$,
where $u \ge 0$ is minimal for which at least
one of $\ k-1$, $(n-3)_{k+1}$ is a vertex in $C+u$,
and $v \ge 0$ is minimal for which at least one of
$\ 0$, $(n-1)_{k+1}$ is a vertex in $C-v$.
\end{lem}

\begin{proof}
If $C$ avoids vertices $k-1$ and $(n-3)_{k+1}$,
then $C+1$ is a cycle by Lemma 3.5.  The translation
by $1$ can be interated but it cannot go on indefinitely, otherwise
given any vertex $a \in C$, the translate $C+(k-a)$ 
would be a cycle containing $k$,
a contradiction.  Thus the iteration must eventually reach the cycle
$C+u$, after which the process stops, by Lemma 3.5.
Similarly, translation by $-1$ can be iterated, with the process stopping
upon reaching the cycle $C-v$.
\end{proof}

\begin{lem}
For $G(n,k)$, the following statements are equivalent:

(A) There are no cycles.

(B) Both $\ k-1$ and $\ (n-3)_{k+1}$ lie in the tail.

(C) Both $\ 0$ and $\ (n-1)_{k+1}$  lie in the tail.

(D) Both $\ 0$ and $ k-1$  lie in the tail.
\end{lem}

\begin{proof}
Clearly (A) implies (B), (C), and (D).

Conversely, suppose that (A) is false, so that there exists a cycle $C$.
Then by Lemma 3.6, at least one of
$k-1$, $(n-3)_{k+1}$ lies in $C+u$, and at least one of
$0$, $(n-1)_{k+1}$ lies in $C-v$, so (B) and (C) are false.
It remains to prove that (D) is false.  Assuming that $k-1$ lies
in the tail, we must show that $0$ lies in a cycle.
The tail must have an initial segment of the form
\[
k \rightarrow a_2 \rightarrow \dots \rightarrow a_{\ell -1} 
\rightarrow k-1.
\]
Since (B) is false, $(n-3)_{k+1}$ does not lie in the tail.
Thus by Lemma 3.2,  $G(n,k)$ has the cycle
\[
0 \rightarrow  (a_2+1)_{k+1} \rightarrow \dots \rightarrow 
(a_{\ell -1} +1)_{k+1} \rightarrow 0.
\]
\end{proof}

\begin{lem}
Suppose that neither $0$ nor $k-1$ lies in the tail of $G(n,k)$.
Then $G(n,k)$ has a cycle $C'$ containing both $k-1$ and $(n-3)_{k+1}$
and a cycle $C''$ containing both $0$ and $(n-1)_{k+1}$.
Moreover, $C'$ is longer than the tail.
\end{lem}

\begin{proof}
The cycle containing $k-1$ has the form
\[
k-1 \rightarrow a_2 \rightarrow \dots \rightarrow a_{\ell -1} 
\rightarrow k-1.
\]
It must be that $(n-3)_{k+1}$ lies in this cycle, for otherwise, by Lemma 3.2,
the tail would begin
\[
k \rightarrow (a_2 +1)_{k+1} \rightarrow \dots \rightarrow
(a_{\ell -1} +1)_{k+1} \rightarrow 0,
\]
which contradicts the fact that $0$ is not in the tail.

To see that this cycle is longer than the tail, note that if
$(n-3)_{k+1} = a_i$, then the full tail is
\[
k \rightarrow (a_2 +1)_{k+1} \rightarrow \dots \rightarrow
(a_{i} +1)_{k+1} =(n-2)_{k+1},
\]
which is shorter than the cycle.   If on the other hand
$(n-3)_{k+1} = k-1$, then $(n-2)_{k+1} = k$,
wherein the tail has length $0$.

We now consider the cycle containing $0$.   
For some $\ell$, this cycle has the form
\[
0 \rightarrow a_2 \rightarrow \dots \rightarrow a_{\ell -1} 
\rightarrow 0.
\]
It must be that $(n-1)_{k+1}$ lies in this cycle, for otherwise, by Lemma 3.4,
the tail would begin
\[
k \rightarrow (a_2 -1)_{k+1} \rightarrow \dots \rightarrow
(a_{\ell -1} -1)_{k+1} \rightarrow k-1,
\]
which contradicts the fact that $k-1$ is not in the tail.
\end{proof}

\begin{thm}
For any fixed pair $k$, $n$,  all cycles in $G(n,k)$ are contiguous
translates of each other.  In particular, these cycles  have the same length.
\end{thm}

\begin{proof}
Let $C_1$ and $C_2$ be two different cycles in $G(n,k)$.
We will consider three cases: when $0$ is in the tail,
when $k-1$ is in the tail, and when neither $0$ nor $k-1$
is in the tail.

First suppose that $k-1$ lies in the tail.
By Lemma 3.6,  there are nonnegative integers $u_1$ and $u_2$ for which
the translates $C_1 + u_1$ and $C_2 + u_2$ both contain
$(n-3)_{k+1}$, while $C_1 + \alpha$ is a cycle for all 
$0 \le \alpha \le u_1$ and 
$C_2 + \beta$ is a cycle for all
$0 \le \beta \le u_2$.
Since $C_1 + u_1$ and $C_2 + u_2$ are not disjoint, we have
$C_1 + u_1  = C_2 + u_2$, so that 
$C_2 = C_1 + (u_1 - u_2)$.  Without loss of generality, $u_1 \ge u_2$.
Thus $C_2$ is one of the contiguous cycles $C_1 + \alpha$ listed above.

Next suppose that $0$ lies in the tail.
By Lemma 3.6,  there are nonnegative integers $v_1$ and $v_2$ for which
the translates $C_1 - v_1$ and $C_2 - v_2$ both contain
$(n-1)_{k+1}$, while $C_1 - \alpha$ is a cycle for all
$0 \le \alpha \le v_1$ and
$C_2 - \beta$ is a cycle for all
$0 \le \beta \le v_2$.
Since $C_1 - v_1$ and $C_2 - v_2$ are not disjoint, we have
$C_1 - v_1  = C_2 - v_2$, so that
$C_2 = C_1 - (v_1 - v_2)$.  Without loss of generality, $v_1 \ge v_2$.
Thus $C_2$ is one of the contiguous cycles $C_1 - \alpha$ listed above.

Finally suppose that neither $0$ nor $k-1$ lies in the tail.
By Lemma 3.6 and Lemma 3.8, there exist nonnegative 
integers $u_1$ and $u_2$ for which
the translates $C_1 + u_1$ and $C_2 + u_2$ are both equal to
the cycle $C'$ defined in Lemma 3.8,
while $C_1 + \alpha$ is a cycle for all
$0 \le \alpha \le u_1$ and
$C_2 + \beta$ is a cycle for all
$0 \le \beta \le u_2$.
Thus,
$C_2 = C_1 + (u_1 - u_2)$, where without loss of generality, $u_1 \ge u_2$.
Thus $C_2$ is one of the contiguous cycles $C_1 + \alpha$ listed above.
\end{proof}

\section{Properties of the tail in $G(n,k)$}

\begin{lem}
Let $C$ be any cycle in $G(n,k)$.
If $0$ or $k-1$ is in the tail of $G(n,k)$,
then the tail
is at least as long as $C$.
\end{lem}

\begin{proof}
First suppose that $k-1$ is in the tail. Then
the tail has an initial segment of the form
\[
k \rightarrow a_2 \rightarrow \dots \rightarrow a_{\ell -1}
\rightarrow k-1.
\]
This path has length $\ell -1$.
By Lemma 3.7, $(n-3)_{k+1}$ is not in the tail, so by Lemma 3.2,
$G(n,k)$ has the following cycle of length $\ell -1$:
\[
0 \rightarrow (a_2+1)_{k+1}  \rightarrow \dots \rightarrow 
(a_{\ell -1}+1)_{k+1} \rightarrow 0.
\]
Since all cycles have the same length by Theorem 3.9,
the tail is at least as long as $C$.

Now suppose that $0$ is in the tail. Then for some $\ell$,
the tail has an initial segment of the form
\[
k \rightarrow a_2 \rightarrow \dots \rightarrow a_{\ell -1}
\rightarrow 0.
\]
This path has length $\ell -1$.
By Lemma 3.7, $(n-1)_{k+1}$ is not in the tail, so by Lemma 3.4,
$G(n,k)$ has the following cycle of length $\ell -1$:
\[
k-1 \rightarrow (a_2-1)_{k+1}  \rightarrow \dots \rightarrow 
(a_{\ell -1}-1)_{k+1} \rightarrow k-1.
\]
Thus again, the tail is at least as long as $C$.
\end{proof}

\begin{lem}
Suppose that neither $0$ nor $k-1$ lies in the tail of $G(n,k)$.
Let $q$ denote the lengths of the cycles $C'$ and $C''$ defined in
Lemma 3.8.  Write $k=fq+m$, where $f$ is the floor $[k/q]$ and
$0 \le m < q$.   Then $G(n,k)$ consists of the $f$ cycles
\begin{equation} \label{eq: 4.1}
C', C' -1, \dots, C'-(f-1)
\end{equation}
together with a tail $(C' +1)^*$ of length $m$,
where $(C' +1)^*$ denotes the segment of $(C' +1)$ beginning with $k$
and ending with $(n-2)_{k+1}$.
At the same time,
$G(n,k)$ consists of the $f$ cycles
\begin{equation} \label{eq: 4.2}
C'', C'' +1, \dots, C''+(f-1)
\end{equation}
together with a tail $(C'' -1)^*$ of length $m$,
where $(C'' -1)^*$ denotes the segment of $(C'' -1)$ beginning with $k$
and ending with $(n-2)_{k+1}$.
As a consequence, $C' = C'' +(f-1)$ and $(C' +1)^* = (C'' -1)^*$.
\end{lem}

\begin{proof}
Since all cycles have length $q$ and the tail has length $<q$ by Lemma 3.8,
there must be exactly $f$ cycles of length $q$ and a tail of length $m$.
There is no cycle $(C'+1)$, because $C'$ contains $k-1$, and 
adding $1$ to $k-1$ yields the vertex $k$ in the tail.  Thus the
set of cycles must consist of the $f$ contiguous cycles in \eqref{eq: 4.1}.
Similarly, since $C''$ contains $0$, there is no cycle $(C'' -1)$, 
so the set of cycles must
consist of the $f$ contiguous cycles in \eqref{eq: 4.2}.  By Lemma 3.8,
the path
\[
k-1 \rightarrow \dots \rightarrow (n-3)_{k+1}
\]
is part of the cycle $C'$, so the tail must be $(C' +1)^*$.
Similarly, the path
\[
0 \rightarrow \dots \rightarrow (n-1)_{k+1}
\]
is part of the cycle $C''$,
so the tail must be $(C'' -1)^*$.
\end{proof}

Define translates of paths exactly as we did for cycles.
To reiterate, suppose $D$ is a path in $G(n,k)$ of the form
$d_1 \rightarrow \dots \rightarrow d_\ell$.
If for some integer $t$, there is a path
$(d_1+t)_{k+1} \rightarrow \dots \rightarrow (d_\ell+t)_{k+1}$
in $G(n,k)$, we call this path the translate $D+t$.

\begin{lem}
Suppose that $k-1$ is in the tail of $G(n,k)$ but 
$(n-3)_{k+1}$ is not.  Thus, for some $\ell \ge 2$, the tail
must have an initial segment of length $\ell -1$ given by
\begin{equation}\label{eq:4.3}
P \rightarrow k-1,
\end{equation}
where $P$ is an open path of length $\ell -2$ of the form
\[
k \rightarrow a_2 \rightarrow \dots \rightarrow a_{\ell -1}.
\]
Let $C$ denote the cycle of length $\ell -1$ given by
\[
0 \rightarrow (a_2+1)_{k+1}  \rightarrow \dots \rightarrow 
(a_{\ell -1}+1)_{k+1} \rightarrow 0,
\]
which exists by Lemma 3.2.   Then $G(n,k)$ consists of the set of cycles
\begin{equation}\label{eq:4.4}
C, C+1, \dots , C+u,
\end{equation}
with $u$ as defined in Theorem 3.6,
together with a tail of the form
\begin{equation}\label{eq:4.5}
P \rightarrow (P-1) \rightarrow \dots \rightarrow (P -t) \rightarrow (P-t-1)^*,
\end{equation}
where $t \ge 0$ is minimal for which $P-t$ contains $(n-1)_{k+1}$,
and $(P-t-1)^*$ denotes the initial segment of $(P-t-1)$ which terminates
with the vertex $(n-2)_{k+1}$.
\end{lem}

\begin{proof}
Since $C$ contains $0$, there is no cycle $C-1$, so by Theorem 3.9,
the contiguous cycles
in \eqref{eq:4.4} must be the full set of cycles in $G(n,k)$.
By \eqref{eq:4.3} and Lemma 3.3, the tail begins
\[
P \rightarrow (P-1) \rightarrow \dots \rightarrow (P -t).
\]
Finally, since  $(P-t)$  contains $(n-1)_{k+1}$ and 
$((n-1)_{k+1} -1)_{k+1}$ equals the terminal vertex $(n-2)_{k+1}$, 
we have $(P -t) \rightarrow (P-t-1)^*$.
\end{proof}

\begin{lem}
Suppose that $0$ is in the tail of $G(n,k)$ but
$(n-1)_{k+1}$ is not.  Thus, for some $\ell \ge 2$, the tail
must have an initial segment of length $\ell -1$ given by
\begin{equation}\label{eq:4.6}
P \rightarrow 0,
\end{equation}
where $P$ is an open path of length $\ell -2$ of the form
\[
k \rightarrow a_2 \rightarrow \dots \rightarrow a_{\ell -1}.
\]
Let $C$ denote the cycle of length $\ell -1$ given by
\[
k-1 \rightarrow (a_2-1)_{k+1}  \rightarrow \dots \rightarrow 
(a_{\ell -1}-1)_{k+1} \rightarrow k-1,
\]
which exists by Lemma 3.4.   Then $G(n,k)$ consists of the set of cycles
\begin{equation}\label{eq:4.7}
C, C-1, \dots , C-v,
\end{equation}
with $v$ as defined in Theorem 3.6,
together with a tail of the form
\begin{equation}\label{eq:4.8}
P \rightarrow (P+1) \rightarrow \dots \rightarrow (P +t) \rightarrow (P+t+1)^*,
\end{equation}
where $t \ge 0$ is minimal for which $P+t$ contains $(n-3)_{k+1}$,
and $(P+t+1)^*$ denotes the initial segment of $(P+t+1)$ which terminates
with the vertex $(n-2)_{k+1}$.
\end{lem}

\begin{proof}
Since $C$ contains $k-1$, there is no cycle $C+1$, so by Theorem 3.9,
the contiguous cycles
in \eqref{eq:4.7} must be the full set of cycles in $G(n,k)$.
By \eqref{eq:4.6} and Lemma 3.1, the tail begins
\[
P \rightarrow (P+1) \rightarrow \dots \rightarrow (P +t).
\]
Finally, since  $(P+t)$  contains $(n-3)_{k+1}$ and
$((n-3)_{k+1} +1)_{k+1}$ equals the terminal vertex $(n-2)_{k+1}$,
we have $(P +t) \rightarrow (P+t+1)^*$.
\end{proof}

\begin{lem}
For the graph $G(n,k)$, the following are equivalent:

(A) $ k-1$ is in the tail.

(B)  $n_k$ is in the tail.

(C)  $(n-1)_{k+1}$ is in the tail.
\end{lem}

\begin{proof}
By \eqref{eq:2.2} with $i=1$,  (B) and (C) are equivalent,
so it remains to prove that (A) is equivalent to (C).
The result is obvious if there are no cycles, so assume
there is a cycle $C$.

First assume (A).  By Lemma 3.7,  $(n-3)_{k+1}$ cannot lie in the tail,
so that the hypotheses of Lemma 4.3 hold.  Now (C) follows, 
since the segment $P-t$
in \eqref{eq:4.5} contains $(n-1)_{k+1}$. 

Conversely, assume that (A) is false.   We must show that
$(n-1)_{k+1}$ lies in a cycle.  If $0$ is not in the tail,
then $(n-1)_{k+1}$ lies in the cycle $C''$, by Lemma 3.8.
If $0$ lies in the tail, then $(n-1)_{k+1}$ lies
in the cycle $C-v$, where $v$ is defined in Lemma 3.6.
\end{proof}

\section{Relation between $G(n,k)$ and $G(n+1,k)$}

Recall from \eqref{eq:2.2} that $G(n,k)$ has the $k$ edges
\[
E(i,n):=
(i + n -2)_{k+1} \rightarrow  (i + n - 1)_{k}, \quad 1 \le i \le k.
\]
The following lemma relates these to the edges of $G(n+1,k)$.

\begin{lem}
The $k$ edges of $G(n+1,k)$ consist of the $k-1$ edges
\[
E(2,n), E(3,n), \dots, E(k,n)
\]
from $G(n,k)$ together with the additional edge $E(k, n+1)$.
Thus the only edge in $G(n,k)$ which is not also an edge in
$G(n+1,k)$ is 
\[
E(1,n):=
   (n-1)_{k+1} \rightarrow n_k;
\]
this edge is ``replaced" in $G(n+1,k)$ by
\[
E(k,n+1):=
   (n-2)_{k+1} \rightarrow n_k.
\]
\end{lem}

\begin{proof}
The result follows because 
\[
E(i+1,n) = E(i, n+1), \quad  1 \le i \le k-1.
\]
\end{proof}

\begin{lem}
If $C$ is a cycle in $G(n,k)$ which does not contain $n_k$,
then $C$ remains a cycle in $G(n+1,k)$.
\end{lem}

\begin{proof}
By Lemma 5.1, all the edges in $C$ remain edges in $G(n+1,k)$,
since $E(1,n)$ is not an edge in $C$.
\end{proof}

\begin{lem}
Suppose that $G(n,k)$ has a cycle $C$ given by
\begin{equation}\label{eq:5.1}
a_1 \rightarrow a_2 \rightarrow \dots \rightarrow a_\ell \rightarrow a_1
\end{equation}
with $a_1 = n_k$.  
Then the tail in $G(n+1,k)$ is
\begin{equation}\label{eq:5.2}
T \rightarrow a_1 \rightarrow \dots \rightarrow a_\ell,
\end{equation}
where $T$ is the tail
\begin{equation}\label{eq:5.3}
T:= k \rightarrow \dots \rightarrow (n-2)_{k+1}
\end{equation}
in $G(n,k)$.
\end{lem}

\begin{proof}
The last edge in \eqref{eq:5.1} is $E(1,n)$, since
$a_1=n_k$.  Thus $a_\ell$ is the terminal vertex $(n-1)_{k+1}$ 
in $G(n+1,k)$.
Since  $E(k,n+1)$ is an edge in $G(n+1,k)$,
we see that $T \rightarrow a_1$ in $G(n+1,k)$.
Finally, the edges
\[
a_i \rightarrow a_{i+1}, \quad 1 \le i \le \ell-1
\]
in $G(n,k)$ all lie in $G(n+1,k)$ by Lemma 5.1.
\end{proof}

We say that $G(n,k)$'s tail $T$ in \eqref{eq:5.3} {\em will absorb}
$G(n,k)$'s cycle $C$ in \eqref{eq:5.1}
if \eqref{eq:5.2} gives the tail in $G(n+1,k)$.

\begin{lem}
If $k-1$ is not in $G(n,k)$'s tail $T$, then $T$ will absorb a cycle.
\end{lem}

\begin{proof}
The result follows from Lemma 5.3 and Lemma 4.5.
\end{proof}

\begin{lem}
Suppose that $n_k$ lies in the tail $T$ of $G(n,k)$, so that
$T$ must have the form
\begin{equation}\label{eq:5.4}
T:= T_1 \rightarrow n_k \rightarrow \dots \rightarrow (n-2)_{k+1},
\end{equation}
where $T_1$ is the tail in $G(n+1,k)$ given by
\begin{equation}\label{eq:5.5}
T_1:=  k \rightarrow \dots \rightarrow (n-1)_{k+1}.
\end{equation}
Then the cycles in $G(n,k)$ remain cycles in $G(n+1,k)$,
and $G(n+1,k)$ has one additional cycle $C$ given by
\begin{equation}\label{eq:5.6}
C:= n_k \rightarrow \dots \rightarrow (n-2)_{k+1} \rightarrow n_k.
\end{equation}
\end{lem}

\begin{proof}
Cycles in $G(n,k)$ remain cycles in $G(n+1,k)$
by Lemma 5.2.  The last edge in \eqref{eq:5.6}
is $E(k,n+1)$, and in view of \eqref{eq:5.4},
the remaining edges in \eqref{eq:5.6}
are also in $G(n+1,k)$ by Lemma 5.1, since none of these are $E(1,n)$.
\end{proof}

We say that $G(n,k)$'s tail $T$ in \eqref{eq:5.4} {\em will shed}
$G(n,k)$'s cycle $C$ in \eqref{eq:5.6}
if \eqref{eq:5.5} gives the tail in $G(n+1,k)$.

\begin{lem}
If $k-1$ lies in $G(n,k)$'s tail $T$, then $T$ will shed a cycle.
\end{lem}

\begin{proof}
The result follows from Lemma 5.5 and Lemma 4.5.
\end{proof}

\begin{lem}[Cf. Prop. 2.2 in \cite{Price1}]
The graph $G(n+1,k)$ has either one more or one less cycle than $G(n,k)$.
\end{lem}

\begin{proof}
Let $T$ denote the tail in $G(n,k)$.
If $k-1$ is not in $T$, then by Lemma 5.4, $T$ will absorb a cycle,
so $G(n+1,k)$ has one less cycle than $G(n,k)$.
If $k-1$ lies in $T$, then by Lemma 5.6, $T$ will shed a cycle,
so $G(n+1,k)$ has one more cycle than $G(n,k)$.
\end{proof}

\section{Line graph connecting points $(n, N(n,k))$}

For fixed $k$, consider the Cartesian graph with $n$ on the horizontal axis
and $N(n,k)$ on the vertical axis.  Create a line graph connecting the
points $(n, N(n,k))$.  
Since $G(n,k)$ depends only on the value of $n$ modulo $k^2+k$,
we restrict our line graph to values of $n$ between $0$ and $k^2+k$.

In this section,
we will show that the union of this line graph with
the horizontal axis has the multimodal shape of a chain of adjoining isosceles
right triangles whose hypotenuses sit on the horizontal axis. 
This shape is actually a consequence of a general result of Price et al.
\cite[Cor. 2.4]{Price1}. (One could see this by viewing $A(n,k)$ as a matrix
over a field of $p$ elements, where $p > k(k+1)$ is prime, and then 
showing that the nullity with respect to this 
underlying finite field is the same as $N(n,k)$.)
However, instead of appealing directly to \cite[Cor. 2.4]{Price1},
we characterize the shape of the line graph using properties of
$G(n,k)$, because this enables us in Section 7 to explicitly determine the
coordinates of all the local peaks.

The multimodal shape of the line graph
is illustrated in Figure 1 for $k=6$, $ 0 \le n \le 42$.
Each dot on the horizontal axis indicates a point where the nullity
is zero, and each dot at the apex of a triangle indicates a point
where the nullity has a local peak.

\begin{figure}[h]
\centering
\scalebox{1.4}
{\includegraphics{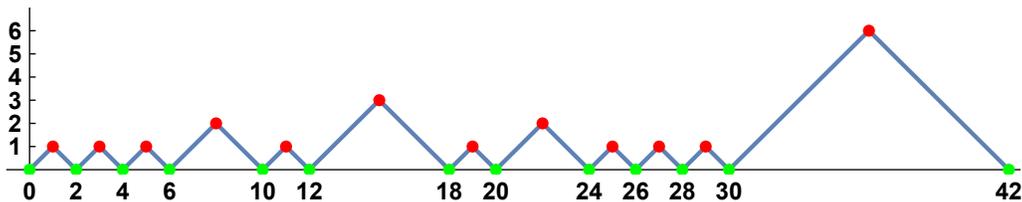}}
\caption{Line graph for k=6}
\end{figure}

The figure shows that the apex at $n=36$ has the maximum height $6$.
In general, $N(k^2,k)=k$, which is the special case $t=k, i=0$ of
Theorem 8.7.

When the line graph for $k=6$ is extended for all $n$ on the horizontal
axis, one sees the symmetry about the vertical line $n=15$.  
For general $k$, the symmetry is about $n = (k^2-k)/2$.
This can
be explained by the fact that the graph $G(n,k)$ is isomorphic to the
graph $G(k^2-k-n,k)$, via the isomorphism that fixes the vertex $k$
and takes every other vertex 
$a$ in $G(n,k)$ to the vertex $k-1-a$ in $G(k^2-k-n,k)$.

To describe the shape of the line graph for general fixed $k$, we  
start at any point $(n,0)$ on the
horizontal axis where the nullity is zero (for example, the origin). 
The corresponding graph $G(n,k)$ has no cycles and has a tail
$T_0$  of length $k$ that begins with a path
\[
B:= k \rightarrow \dots \rightarrow k-1,
\]
since $k-1$ lies in $T_0$.
Let $q$ denote the length of $B$ and write $k=fq+m$,
where $f = [k/q]$ and $0 \le m <q$.  Since $k-1$ is in $T_0$,
it follows from Lemma 5.6 that $T_0$ will shed some cycle $C$
in $G(n+1,k)$.  In $G(n+1,k)$, the translate $B+1$ is a cycle
containing $0$, which can be seen by replacing $n$ by $n+1$
in all of the edges $E(i,n)$.
Since $G(n+1,k)$ has only one cycle, we see that $C=B+1$,
so that $C$ is a cycle of length $q$ that contains $0$.
The tail $T_1$ of $G(n+1,k)$ has length $k-q$.  As we continue
incrementing the first argument in $G$, 
we now argue that the process of shedding can continue,
so that for each $\omega \in [1,f]$, 
the graph $G(n+\omega,k)$ has
$\omega$ contiguous cycles
\begin{equation}\label{eq:6.1}
C, C+1, \dots, C+\omega -1
\end{equation}
along with a tail $T_\omega$ of length $k - \omega q$.
(Note that $C-1$ can never be a cycle, since it contains the vertex $k$.)
The process of shedding was able to continue as long as $\omega <f$,
because then $k-1$ must be in the tail $T_\omega$, otherwise 
$T_\omega$ would be shorter than $q$ by Lemma 3.8.
The graph $G(n+\omega,k)$ has the cycles in \eqref{eq:6.1}
by Lemma 4.5 and Lemma 5.2.

For $\omega = f$, the tail $T_f$ has length $m <q$, so it is too
short to be able to shed a cycle of length $q$.   
For brevity, write $h=n+f$.
Thus, for the graph $G(h,k)$,
$k-1$ is in the cycle $C+f-1$ rather than in the tail $T_f$.
Therefore, by Lemma 3.8 and Lemma 4.2 with $n=h$,  the cycles 
$C$ and $C+f-1$  in $G(h,k)$ are
$C''$ and $C'$, respectively, where $C''$ contains both
$0$ and $(h-1)_{k+1}$, and $C'$ contains both $k-1$ and $(h-3)_{k+1}$.

We have so far shown that 
\begin{equation}\label{eq:6.2}
N(n+\omega, k) = \omega, \quad 0 \le \omega \le f.
\end{equation}
The points
\[
(n+\omega, N(n+\omega, k)) = (n +\omega,\omega),  \quad 0 \le \omega \le f 
\]
all lie on line segment of slope $1$ connecting
the lowest point $(n,0)$ 
to the highest point $(h, f)$. This line segment turns out to be the left
leg of an isosceles right triangle, as we now show.

The graph $G(h,k)$ has $f$ cycles, i.e. $N(h,k) = f$.
It remains to prove that
\begin{equation}\label{eq:6.3}
N(h+\omega, k) = f-\omega, \quad 1 \le \omega \le f,
\end{equation}
which will show that the points
\[
(h+\omega, N(h+\omega, k)) = (h +\omega,f-\omega),  \quad 0 \le \omega \le f
\]
all lie on the right leg of the isosceles right triangle,
with the highest point $(h,f)$ situated at the apex
and the lowest point $(h+f, 0)$ situated on the horizontal axis.

We proceed to prove \eqref{eq:6.3} by showing that as one 
continues to increment the
first argument of $G$, the $f$ cycles of $G(h,k)$  
will get absorbed one by one until there are no cycles left.
The cycle $C''$ in $G(h,k)$ contains $(h-1)_{k+1}$,
so it must also contain $h_k$, since 
\[
E(1,h):= (h-1)_{k+1} \rightarrow h_k.
\]
Thus by Lemmas 5.3 and 5.4,
the tail $T_f$ in $G(h,k)$ will absorb the cycle $C''$.
Now $G(h+1,k)$ consists of a tail $T_{f+1}$ of length $m+q$
together with the ${f-1}$ contiguous cycles 
\[
C''+1, C''+2, \dots, C''+f-1.
\]
Recall that  
$C'' +f-1$ is the same as the cycle $C'$, which contains $k-1$.
As the process of absorption continues, $C'$ has to be the last cycle
absorbed, i.e., $C'$ cannot be absorbed while there are still other cycles
remaining.  This is because once the tail contains both $0$ and $k-1$,
there can be no cycles left, by Lemma 3.7.
By Lemma 5.4, the process of absorption continues as long as $(k-1)$
is not in the tail.
Thus, for $1 \le \omega \le f$, the graph $G(h+\omega,k)$ will have
$f-\omega$ contiguous cycles
\[
C''+\omega , \dots, C''+f -1
\]
along with a tail $T_{f+\omega}$ of length $m + \omega q$.
In particular, $G(h+f,k)$ has a tail of length $k$ and no cycles.
This completes the proof of \eqref{eq:6.3}.

\section{Determination of the nullities $N(n,k)$}

If we knew the coordinates of the apex of a triangle in the line graph,
we would know the coordinates of all points on the legs of the triangle.
For example, when $k=50$, there is an apex at point $(878,4)$, signifying that
when $n=878$, the nullity is $4$, a local peak.   The line graph then
tells us that as
$n$ ranges in order from $874$ to $882$, the nullities are
$0,1,2,3,4,3,2,1,0$, respectively.  Thus, to determine all the nullities,
it remains to determine the coordinates of the apexes.
This will be accomplished in Theorem 7.5.

For an inequality $Q$,  define $\chi(Q)$ to be $1$ or $0$ according as
$Q$ is true or false.  We begin with three technical lemmas evaluating
the sizes of certain sets of integers $a$.  
In these lemmas, $M$, $j$, $q$, and $y$
are positive integers.  Throughout the sequel, for an integer $\ell$, let
$(\ell)_q$ denote the least nonnegative residue of $\ell$ modulo $q$.

\begin{lem}
Let $M,j \in [1,q]$ and $y \in [j,q]$.   Then
\begin{equation}\label{eq:7.1}
|\{a \in [0,M-1]: y-j \le  (aj)_q <y\}| = [Mj/q] + \chi(y \le (Mj)_q).
\end{equation}
\end{lem}

\begin{proof}
The result is easily checked for $M=1$, so assume as induction hypothesis
that \eqref{eq:7.1} holds for some $M \in [1,q)$.   Then
\begin{equation*}
\begin{split}
&|\{a \in [0,M]: y-j \le  (aj)_q <y\}| \\
&=[Mj/q] + \chi(y \le (Mj)_q) + \chi(y-j \le (Mj)_q <y) \\
&=[Mj/q] + \chi(y-j \le (Mj)_q).
\end{split}
\end{equation*}
To complete the induction, it remains to prove that
\begin{equation}\label{eq:7.2}
[Mj/q] + \chi(y-j \le (Mj)_q) = [(M+1)j/q] + \chi(y \le ((M+1)j)_q).
\end{equation}

Since $x = q[x/q] + x_q$ for any positive integer $x$, we have
\begin{equation}\label{eq:7.3}
Mj = \alpha q + \beta, \mbox{with}  \ \ \alpha=[Mj/q], \ \ \beta = (Mj)_q
\end{equation}
and
\begin{equation}\label{eq:7.4}
((M+1)j)_q = (\beta +j)_q = (\beta+j) - q[(\beta +j)/q]. 
\end{equation}
By \eqref{eq:7.3} and \eqref{eq:7.4}, we can write \eqref{eq:7.2}
in the form
\begin{equation}\label{eq:7.5}
\alpha+ \chi(y \le \beta+j)=\alpha + [(\beta+j)/q]
+\chi(y \le (\beta+j) -q[(\beta+j)/q]).
\end{equation}
When $\beta+j < q$, both sides of \eqref{eq:7.5} match, since
$[(\beta+j)/q]=0$.  When $\beta+j \ge q$, we have $[(\beta+j)/q]=1$,
so that \eqref{eq:7.5} becomes
\begin{equation}\label{eq:7.6}
\alpha +1 = \alpha +1 + \chi(y \le \beta +j -q).
\end{equation}
Since $y \ge j > \beta +j -q$, the rightmost term in \eqref{eq:7.6}
vanishes, so that \eqref{eq:7.6} holds.  This completes the proof
of \eqref{eq:7.5}.
\end{proof}

\begin{lem}
Let $M,j \in [1,q]$ and $y \in [1,j]$.   Then
\begin{equation}\label{eq:7.7}
|\{a \in [0,M-1]: y \le  (aj)_q <y + q - j\}| = 
M -[Mj/q] - \chi(y \le (Mj)_q).
\end{equation}
\end{lem}

\begin{proof}
Both sides vanish when $j=q$, so we may assume $j < q$.
The result is easily checked for $M=1$, so assume as induction hypothesis
that \eqref{eq:7.7} holds for some $M \in [1,q)$.   Then
\begin{equation*}
\begin{split}
&|\{a \in [0,M]: y \le  (aj)_q <y +q - j\}| \\
&=M-[Mj/q] - \chi(y \le (Mj)_q) + \chi(y \le (Mj)_q <y+q - j) \\
&=M-[Mj/q] - \chi(y+q-j \le (Mj)_q).
\end{split}
\end{equation*}
To complete the induction, it remains to prove that
\begin{equation}\label{eq:7.8}
-[Mj/q] - \chi(y+q-j \le (Mj)_q) = 1-[(M+1)j/q] - \chi(y \le ((M+1)j)_q).
\end{equation}
The proof of \eqref{eq:7.8} is analogous to that of \eqref{eq:7.2},
so we omit the details.
\end{proof}

When $y=0$, \eqref{eq:7.7} holds if the rightmost term is omitted,
as is shown in the following lemma.
\begin{lem}
Let $M,j \in [1,q]$.   Then
\begin{equation}\label{eq:7.9}
|\{a \in [0,M-1]: 0 \le  (aj)_q < q - j\}| = 
M -[Mj/q].
\end{equation}
\end{lem}

\begin{proof}
This follows by induction on $M$, as with the preceding two lemmas.
\end{proof}

Given an edge $E: = a \rightarrow b$ in $G(n,k)$, define its length to be
$|E| = (b-a)_{k+1}$. There are only two possible edge lengths, as is shown
in the next lemma.

\begin{lem}
Given an edge $E: = a \rightarrow b$ in $G(n,k)$, its length is
\begin{equation}\label{eq:7.10}
|E| =
\begin{cases} (n_k - n_{k+1} +2)_{k+1}
& \mbox{   if} \ \ \ b \le n_k -1  \\
(n_k - n_{k+1} +1)_{k+1}
& \mbox{   if} \ \ \ b \ge n_k .
\end{cases}
\end{equation}
\end{lem}

\begin{proof}
For some $i \in [1, k]$ we have 
\[
a = (n+i-2)_{k+1}, \quad b= (n+i-1)_k, \quad |E|=(b-n_{k+1}-i +2)_{k+1}.
\]
If $b \le n_k-1$, then $b = n_k +i - 1 -k$, so that
\[
|E| = (n_k +i-1-k -n_{k+1}-i+2)_{k+1} = (n_k - n_{k+1} +2)_{k+1}.
\]
If $b \ge n_k$, then $b = n_k +i - 1$, so that
\[
|E| = (n_k +i-1 -n_{k+1}-i+2)_{k+1} = (n_k - n_{k+1} +1)_{k+1}.
\]
\end{proof}

In the first case of \eqref{eq:7.10}, we call $E$ a long edge,
while in the second case, we call $E$ a short edge.

Fix $k$. In preparation for Theorem 7.5, we introduce the following notation.
Let $q$ be an integer in the interval $[1,k]$.  From the division algorithm,
\begin{equation}\label{eq:7.11}
k = fq + m, \ \ \mbox{where} \ \ f=f(q):= [k/q], \quad m =m(q):= k_q.
\end{equation}
Let $j  \in [1,q]$ be an integer 
which is relatively prime to $q$.
Define
\begin{equation}\label{eq:7.12}
s(q,j) := |\{a \in [1,m]: (aj)_q \le (Mj)_q\}| \ , \mbox{where} \ M:=m+1.
\end{equation}
For integer $r \in [1,q-1]$, define
\begin{equation}\label{eq:7.13}
t(r,q,j):= |\{a \in [0,m]: (aj)_q < (rj)_q\}|.
\end{equation} 
Note that 
\[
0 \le s(q,j) \le m, \quad  1 \le t(r,q,j) \le M.
\]
Let $\eta$ denote the specific value of $n$ defined by
\begin{equation}\label{eq:7.14}
\eta = \eta(q,j): = k[(k+1)j/q] -k + (Mj)_q f +s(q,j).
\end{equation}
It's not difficult to check that  $\eta(q,j) \in [f,k^2+k-f]$.
For each integer $w \in [1, f]$, define
\begin{equation}\label{eq:7.15}
c(r) = c(r,w,q,j): =-1-w + (rj)_q f + t(r,q,j), \quad 1 \le r \le q-1,
\end{equation}
and set
\begin{equation}\label{eq:7.16}
c(0) = c(q) := k-w.
\end{equation}
It is  easy to check that for  each $w \in [1, f]$, the integers 
$c(1), \dots, c(q-1)$
are distinct  and lie in the interval $[0, k-f-1]$.

Reducing modulo $k$ and modulo $k+1$ in \eqref{eq:7.14}, we have
\begin{equation}\label{eq:7.17}
\eta_k = (Mj)_q f +s(q,j)
\end{equation}
and
\begin{equation}\label{eq:7.18}
\eta_{k+1} \equiv -[(k+1)j/q] +1 +(Mj)_q f +s(q,j) \pmod {k+1}.
\end{equation}
Thus by Lemma 7.4, the long edges in $G(\eta, k)$ have length
\begin{equation}\label{eq:7.19}
(\eta_k - \eta_{k+1} +2)_{k+1}=
\begin{cases} 1
& \mbox{   if} \ \ \ q=1  \\
[(k+1)j/q]+1
& \mbox{   if} \ \ \ q>1 .
\end{cases}
\end{equation}
and the short edges in $G(\eta, k)$ have length
\begin{equation}\label{eq:7.20}
(\eta_k - \eta_{k+1} +1)_{k+1}=
\begin{cases} 0
& \mbox{   if} \ \ \ q=1  \\
[(k+1)j/q]
& \mbox{   if} \ \ \ q>1 .
\end{cases}
\end{equation}
Observe that
\begin{equation}\label{eq:7.21}
[(k+1)j/q] = fj + [Mj/q].
\end{equation}

We are now prepared to identify the coordinates of the apexes
in our line graph.

\begin{thm}
The graph $G(\eta(q,j), k)$ has 
$f$ distinct cycles of length $q$ given by
\begin{equation}\label{eq:7.22}
k-w=c(0) \rightarrow c(1) \rightarrow \dots \rightarrow c(q)=k-w, \ \ \ 
\quad w \in [1,f],
\end{equation}
so that the line graph
has a local peak at each point in the set
\[
\S:=\{(\eta(q,j), f): 1 \le q \le k,  \ 1 \le j \le q, \ \gcd(j,q)=1\}.
\]
Moreover, the points in $\S$ comprise the totality of local peaks in
our line graph.
\end{thm}

\begin{proof}

We begin by proving that for each $r \in [1,q-1]$,
\begin{equation}\label{eq:7.23}
c(r) \le \eta_k -1 \ \ \ \mbox{if and only if} \ \ \ (rj)_q \le (Mj)_q.
\end{equation}
First suppose that $(rj)_q \le (Mj)_q$.
Subtracting \eqref{eq:7.15} from \eqref{eq:7.17}, we have
\begin{equation}\label{eq:7.24}
\begin{split}
&\eta_k -c(r) =  \\ 
&((Mj)_q - (rj)_q)f+w+|\{a \in [1,m]: (rj)_q \le (aj)_q \le (Mj)_q \}|. 
\end{split}
\end{equation}
The right side of \eqref{eq:7.24} is
greater or equal to $0+w+0 = w \ge 1$; thus \eqref{eq:7.23} holds in this case.

Next suppose that $(rj)_q > (Mj)_q$.
Subtracting \eqref{eq:7.17} from \eqref{eq:7.15}, we have
\begin{equation}\label{eq:7.25}
\begin{split}
&c(r) - \eta_k  =  \\ 
&((rj)_q - (Mj)_q)f-w+|\{a \in [1,m]: (Mj)_q < (aj)_q < (rj)_q \}|. 
\end{split}
\end{equation}
The right side of \eqref{eq:7.25} is
greater or equal to $f -w +0 \ge 0$.   
This completes the proof of \eqref{eq:7.23}.

For $r=q$, \eqref{eq:7.23} fails to hold; instead we have
\begin{equation}\label{eq:7.26}
c(q) = k - w \ge \eta_k.
\end{equation}
This follows because 
\[
w + \eta_k \le f + \eta_k =(1 + (Mj)_q)f + s(q,j) \le qf +m = k.
\]

We proceed to prove that \eqref{eq:7.22} is a cycle in $G(\eta, k)$
for each $w \in [1,f]$.
First consider the case $q=1$, wherein $f=k$, $m=0$, and $\eta = k^2$.
By \eqref{eq:2.2}, all $k$ edges in $G(k^2, k)$ are loops
of the form $(i-1)_{k+1} \rightarrow (i-1)_{k+1}$.
Thus, when $q=1$, \eqref{eq:7.22} is a cycle in $G(\eta, k)$
(of length $1$) for each $w \in [1,f]$.

Now assume that $q > 1$.  By \eqref{eq:7.26} and Lemma 7.4, 
an edge in $G(\eta, k)$ ending with $c(q) = k-w$ must be a short edge.
Therefore, in order to conclude that $c(q-1) \rightarrow c(q)$ is an edge
in $G(\eta, k)$, it suffices to show its edge length is given by \eqref{eq:7.21},
i.e., it suffices to show that
\begin{equation}\label{eq:7.27}
(c(q) - c(q-1))_{k+1} = fj +[Mj/q].
\end{equation}
The left member of \eqref{eq:7.27} equals
\begin{equation*}
\begin{split}
& (k-w + 1 + w +(j-q)f - t(q-1,q,j))_{k+1} = \\
&((j - q)f - |\{a \in [0,m]: (aj)_q < q-j \}|)_{k+1} = \\
& (jf + M -(M - [Mj/q]))_{k+1} = fj + [Mj/q],
\end{split}
\end{equation*}
where the penultimate equality follows from Lemma 7.3.
This completes the proof of \eqref{eq:7.27}.

To prove that \eqref{eq:7.22} is a cycle in $G(\eta,k)$,
it remains to show that
$E_r:=c(r-1) \rightarrow c(r)$ 
is an edge in $G(\eta, k)$ for each $r \in [1,q-1]$.
To this end, it suffices to show, in view of
 \eqref{eq:7.23} and Lemma 7.4,
 that 
\begin{equation}\label{eq:7.28}
|E_r|:= (c(r)-c(r-1))_{k+1}=jf + [Mj/q] + \chi((rj)_q \le (Mj)_q)
\end{equation}
for each $r \in [1,q-1]$.

\noindent{\em Case 1:} $\ (rj)_q > ((r-1)j)_q$.

We have
\begin{equation*}
|E_r| = ((rj)_q  - ((r-1)j)_q) f +|\{a \in [0,m]:((r-1)j)_q \le (aj)_q < (rj)_q \}|.
\end{equation*}
In Case 1,
\[
((r-1)j)_q = (rj)_q - j,
\]
so that
\[
|E_r| = jf +|\{a \in [0,m]: (rj)_q - j \le (aj)_q < (rj)_q \}|.
\]
Applying Lemma 7.1 with $y = (rj)_q$, we deduce \eqref{eq:7.28}.

\noindent{\em Case 2:} $\ (rj)_q < ((r-1)j)_q$.

Since 
\[
(c(r) - c(r-1))_{k+1} = k+1 - (c(r-1) - c(r))_{k+1},
\]
we have
\begin{equation*}
|E_r| = k+1 + ((rj)_q  - ((r-1)j)_q) f -|\{a \in [0,m]:(rj)_q \le (aj)_q < ((r-1)j)_q \}|.
\end{equation*}
In Case 2,
\[
((r-1)j)_q = (rj)_q +q - j,
\]
so that
\[
|E_r| = jf + M - |\{a \in [0,m]: (rj)_q  \le (aj)_q < (rj)_q +q -j\}|.
\]
Applying Lemma 7.2 with $y = (rj)_q$, we deduce \eqref{eq:7.28}.

The remark below \eqref{eq:7.16} shows that the cycles in \eqref{eq:7.22}
are distinct.  Thus $G(\eta(q,j),k)$ has $f$ cycles of length $q$, which
proves 
that every point in $\S$ is a local peak.

We proceed to prove that the integers 
\[
\eta(q,j) \ \ \mbox{with} \quad  
1 \le q \le k,  \ \ 1 \le j \le q, \ \ \gcd(j,q)=1
\]
are all distinct, so that
\begin{equation}\label{eq:7.29}
|\S| = \sum_{q=1}^k \phi(q),
\end{equation}
where $\phi$ is Euler's totient function.
Suppose that $\eta(q,j) = \eta(q',j')$.
Then $G(\eta(q,j),k)= G(\eta(q',j'),k)$.
Since the cycles in $G(\eta(q,j),k)$ have length $q$ for each $j$, we must have
$q = q'$.  Using \eqref{eq:7.19} or \eqref{eq:7.20} to compare edge lengths, we have
$[(k+1)j/q] = [(k+1)j'/q]$, which forces $j = j'$.   This completes the proof 
of \eqref{eq:7.29}.

Each point $(\eta(q,j),f) \in \S$ is an
apex of an isosceles right triangle in our line graph
whose base has length $2f = 2[k/q]$.  For each $q \in [1,k]$, 
there are $\phi(q)$ values of $j$ for which the 
base of the corresponding triangle has length $2[k/q]$.
Summing the base lengths of all the triangles with apexes in $\S$ yields the total length
\[
2 \sum_{q=1}^k \phi(q) [k/q].
\]
Amazingly, this sum equals $k^2 +k$ \cite[Ex. 9, p. 29]{Rose}.
Thus there is no room in the line graph for a triangle 
whose apex is not in $\S$.
This completes the proof that the points in $\S$ comprise the totality of local peaks
in our line graph.
\end{proof}

\section{Applications of the formulas for $N(n,k)$}

For any fixed nonnegative integer $z$, let $P_z(k)$ denote the percentage 
of matrices $A(n,k)$ with $n \in [0,k^2+k)$
for which $A(n,k)$ has nullity  $z$.   Theorems 8.1 and 8.2
provide asymptotic formulas for $P_z(k)$
as $k \rightarrow \infty$. Theorem 8.1 starts off with the case $z=0$.

\begin{thm}
Let $P_0(k)$ denote the percentage of matrices $A(n,k)$ which are nonsingular.
As $k$ gets large,  $P_0(k)$ approaches $3/\pi^2 \simeq 30.4 \%$.
\end{thm}

\begin{proof}
For each fixed $k$ and $n \in [0,k^2+k)$, 
$G(n,k)$ has no cycles if and only if $n$ is the left endpoint of the
hypotenuse of a triangle in the line graph.
Since there are $|\S|$ triangles in the line graph,
it follows that as $k$ gets large,  $P_0(k)$ is asymptotic to  $|\S|/(k^2+k)$.   
By \eqref{eq:7.29} and \cite[p. 26]{Rose},
$|\S|$ is asymptotic to $3 k^2 /\pi^2$.
Thus $P_0(k)$ approaches $3/\pi^2$.
\end{proof}

\begin{thm}
Fix an integer $z \ge 1$.
Let $P_z(k)$ denote the percentage of matrices $A(n,k)$ which have nullity $z$.
As $k$ gets large,  $P_z(k)$ approaches $3(1/z^2 + 1/(z+1)^2)/\pi^2$.
\end{thm}

\begin{proof}
We argue as in the previous proof.  The number of triangles of height $z$ 
in the line graph
is
\[
\sum_{k/(z+1) < q \le k/z} \phi(q),
\]
and each of these triangles contains a single point of height $z$, namely the apex.
The number of triangles of height $>z$ is
\[
\sum_{1 \le q \le k/(z+1)} \phi(q),
\]
and each of these triangles contains two points of height $z$.
Thus there are
\[
\sum_{k/(z+1) < q \le k/z} \phi(q) +2 \sum_{1 \le q \le k/(z+1)} \phi(q)
= \sum_{q=1}^{k/z} \phi(q) + \sum_{q=1}^{k/(z+1)} \phi(q)
\]
points of height $z$ in the line graph.
The right side is asymptotic to
\[
\frac{3k^2}{ z^2\pi^2} + \frac{3k^2}{(z+1)^2\pi^2 },
\]
so $P_z(k)$ approaches $3(1/z^2 + 1/(z+1)^2)/\pi^2$.
\end{proof}

For $k=300$, the exact number of $n \in [0, k^2 +k)$  for which
$N(n,k)=0$, $N(n,k)=1$, and $N(n,k)=2$ is $27398$, $34256$, and $9902$,
respectively.  
For comparison, the estimates in Theorems 8.1 and  8.2 give
\[
90300*\cfrac{3}{\pi^2} \simeq 27447.9, \ \  
90300*\cfrac{15}{4\pi^2} \simeq 34309.9,  \ \ 
90300*\cfrac{13}{12\pi^2} \simeq 9911.7.
\]

Fix $k$.   It is not a straightforward
task to evaluate $N(n,k)$ for a random value of $n$, 
since the nullity in Theorem 7.5 is not expressed explicitly as a function 
of $n$.   However, when $n$ is given as a function of $k$, we can often
apply Theorem 7.5 directly to evaluate $N(n,k)$.  The next eight theorems give
illustrative examples.

By \cite[Lemma 1]{EW}, $N(2k-\theta,k)=1$ for all odd $\theta$
with $k-1 \ge \theta \ge -1$.  Lemma 9.1 below shows that
for all even $\epsilon$ with $k-1 \ge \epsilon \ge -2$, we have
\begin{equation*}
N(2k-\epsilon,k)=
\begin{cases} 2
& \mbox{   if} \ \ \ \epsilon \equiv k+1 \pmod{3}  \\
0
& \mbox{   if} \ \ \ \epsilon \not \equiv k+1 \pmod{3} .
\end{cases}
\end{equation*}
A different method is needed to evaluate $N(2k+b, k)$ when $b \ge 3$.
In the next four theorems, we apply
Theorem 7.5 to evaluate $N(2k+b, k)$ for $b \in \{3,4,5,6\}$.

\begin{thm}
$N(2k+3,k)$ equals $3$ or $1$ according as $3 \mid k$ or $3 \nmid k$.
\end{thm}

\begin{proof}
First suppose that $3 \mid k$. Then by \eqref{eq:7.14} with
$q=k/3$ and $j=1$, we have $\eta = 2k+3$,
since $m=0$ and $f=3$  when $q=k/3$.
By Theorem 7.5, $N(2k+3,k) = N(\eta, k) = f =3$, as desired.
Next suppose that $3 \nmid k$.  Then with $q=k$ and $j=3$, we have
$\eta = 2k+3$,
since $m=0$ and $f=1$  when $q=k$.
By Theorem 7.5, $N(2k+3,k) = N(\eta, k) = f =1$, as desired.
\end{proof}

\begin{thm}
$N(2k+4,k)$ equals $2$ or $0$ according as $3 \mid k$ or $3 \nmid k$.
\end{thm}

\begin{proof}
In the proof of Theorem 8.3, the point $(\eta, N(\eta,k)) = (2k+3,f)$
is an apex of a triangle in the line graph,
where $f$ is $3$ or $1$ according as $3 \mid k$ or $3 \nmid k$.
The neighboring point to the right of the apex on the line graph is 
$(2k+4, f-1)$, so that  $N(2k+4,k)$ equals
$2$ or $0$ according as $3 \mid k$ or $3 \nmid k$.
\end{proof}

\begin{thm}
$N(2k+5,k)=1$  for all $k \ge 1$.
\end{thm}

\begin{proof}
The proof of Theorem 8.3 yields the apex  $(2k+3,3)$  or $(2k+3,1)$
according as $3 \mid k$ or $3 \nmid k$. Moving two units to the right
along the line graph 
yields the point $(2k+5,1)$ in either case.  Thus $N(2k+5,k)=1$.
\end{proof}

\begin{thm}
For $k >2$, $N(2k+6,k)$ equals $2$ or $0$ according as
$k \equiv 1 \pmod{3}$ or $k \not \equiv 1 \pmod{3}$.
\end{thm}

\begin{proof}
First suppose that $3 \mid k$.  As was noted in the proof of
Theorem 8.4, the point $(2k+3,3)$ is an apex of a triangle in the
line graph.   Moving down the right leg of this triangle, we see
that $(2k+6, 0)$ is on its base, so $N(2k+6,k)=0$. 

Next suppose that $k \equiv 1 \pmod{3}$. To show that $(2k+6,2)$
is on the line graph, it suffices to show that $(2k+7,3)$ is an apex
of a triangle.   By \eqref{eq:7.14} with
$q=(k-1)/3$ and $j=1$, we have $\eta = 2k+7$,
since $m=1$ and $f=3$  when $q=(k-1)/3$.
Thus by Theorem 7.5, $(\eta,f) = (2k+7,3)$ is an apex.

Finally suppose that $k \equiv 2 \pmod{3}$.
This is the trickiest case.
To show that $(2k+6,0)$
is on the line graph, it suffices to show that $(2k+5,1)$ is an apex
of a triangle.
From \eqref{eq:7.14} with
$q=(2k-1)/3$ and $j=2$, we can prove that $\eta = 2k+5$,
using the facts that  $m=(k+1)/3$ and $f=1$  when $q=(2k-1)/3$.
The proof that $\eta = 2k+5$ is facilitated by
the observations that $(2M)_q = 3$ and $s(q,2) = 2$,
where the last equality follows because $a=1$ and $a = m$
are the only values of $a$ satisfying the inequality in \eqref{eq:7.12}.
Thus by Theorem 7.5, $(\eta,f) = (2k+5,1)$ is an apex.
\end{proof}

\begin{thm}
Let $t, k \ge 2$ and suppose that $(t+1) \mid (k+1)$.
Then 
\[
N(tk + i, k) = N(tk -i,k) = t-i, \quad 0 \le i \le t.
\]
\end{thm}

\begin{proof}
By \eqref{eq:7.14} with
$q=(k+1)/(t+1)$ and $j=1$, we have $\eta = tk$,
since $m=(k-t)/(t+1)$ and $f=t$  when $q=(k+1)/(t+1)$.
Thus by Theorem 7.5, $(\eta,f) = (tk,t)$ is an apex,
which proves the result for $i=0$.   For general $i \in [0,t]$,
the result follows by descending from the apex 
along the line graph.
\end{proof}

\begin{thm}
For $1 < c < k$, we have 
$N(ck +c-k, k) = \gcd(c,k)$.
\end{thm}

\begin{proof}
Write $d = \gcd(c,k)$.
By \eqref{eq:7.14} with
$q=k/d$ and $j=c/d$, we have $\eta = ck+c-k$,
since $m=0$ and $f=d$  when $q=k/d$.
Thus by Theorem 7.5, $(\eta,f) = (ck+c-k,d)$ is an apex,
and $N(ck +c-k, k) = d$.
\end{proof}

\begin{thm}
When $k$ is even, $A(3k,k)$ is nonsingular, i.e., $N(3k,k)=0$.  
When $k$ is odd,
$N(3k,k)$ equals $1$ or $3$ according as
$k \equiv 1 \pmod{4}$ or $k \equiv 3 \pmod{4}$.
\end{thm}

\begin{proof}
First suppose that $k \equiv 3 \pmod{4}$.
Then $N(3k, k) = 3$ by the special case $t=3, i=0$ of Theorem 8.7.

Next suppose that $k \equiv 1 \pmod{4}$.
By \eqref{eq:7.14} with
$q=(k+1)/2$ and $j=2$, we have $\eta = 3k$,
since $m=(k-1)/2$ and $f=1$  when $q=(k+1)/2$.
Thus by Theorem 7.5, $(\eta,f) = (3k,1)$ is an apex,
and $N(3k, k) = 1$.

Next suppose that $k \equiv 0 \pmod{4}$.
By \eqref{eq:7.14} with
$q=k/4$ and $j=1$, we have $\eta = 3k+4$,
since $m=0$ and $f=4$  when $q=k/4$.
Thus by Theorem 7.5, $(\eta,f) = (3k+4,4)$ is an apex of a triangle,
so that $N(3k+4, k) = 4$.  Moving four units to the left of the apex,
we obtain the point $(3k,0)$ on the base of the triangle, so that
$N(3k,k)=0$.

Finally suppose that $k \equiv 2 \pmod{4}$.
It is straightforward to check that $A(6,2)$ is nonsingular,
so we may assume that $k \ge 6$.
By \eqref{eq:7.14} with
$q=(k+2)/4$ and $j=1$, we have $\eta = 3k-3$,
since $m=(k-6)/4$ and $f=3$  when $q=(k+2)/4$.
Thus by Theorem 7.5, $(\eta,f) = (3k-3,3)$ is an apex of a triangle,
so that $N(3k-3, k) = 3$.  Moving three units to the right of the apex,
we obtain the point $(3k,0)$ on the base of the triangle, so that
$N(3k,k)=0$.
\end{proof}

\begin{thm}
Let $a,b$ be integers with
$b \ge 1$ and $0 \le a \le b$.
Then
\[
N((k+1 +a-b)(k-a)/b, k) = (k-a)/b, \quad \mbox{if} \ \ k \equiv a \pmod{b}.
\] 
\end{thm}

\begin{proof}
First suppose that $a <b$.  Then by \eqref{eq:7.14} with
$b=q$ and $j=1$, we have $\eta = (k+1 +a-b)(k-a)/b$,
since $m=a$ and $f = (k-a)/b$.  Since $(\eta, f)$ is an apex
by Theorem 7.5, the result follows in this case.
Now suppose that $a=b$.  We must prove that
\[
N(-1-k +(k+1)k/b, k) = -1  + k/b.
\]
By \eqref{eq:7.14} with $b=q$ and $j=1$, we have
$\eta =-k +(k+1)k/b$, since $m=0$ and $f=k/b$.
Since $(-k + (k+1)k/b, k/b)$ is an apex by Theorem 7.5,
the result follows by moving one unit to the left of the apex
along the line graph.
\end{proof}

For example, take $b=4$.  Then we have the nullity formulas
\begin{align*}
N((k^2-3k)/4, k) &= k/4, \quad &\mbox{if} \ \  k \equiv 0 \pmod{4},\\ \\
N( (k^2 -3k+2)/4, k) &= (k-1)/4, \quad &\mbox{if} \ \ k \equiv 1 \pmod{4},\\ \\
N( (k^2-3k+2)/4, k) &= (k-2)/4, \quad &\mbox{if} \ \ k \equiv 2 \pmod{4},\\ \\
N((k^2-3k)/4, k) &= (k-3)/4, \quad &\mbox{if} \ \ k \equiv 3 \pmod{4}.
\end{align*}

\section{Relation between $D(n,k,x)$ and $N(n,k)$}

Recall from the spectral theorem \cite[Theorem 12.2.2]{GH}
that $A(n,k)$ is unitarily diagonalizable.   In particular,
its nonzero eigenvalues are purely imaginary and occur in complex
conjugate pairs, so that $A(n,k)$ has even rank.

In the Introduction, we conjectured that for even $n$, the zero $x=0$ of the
polynomial  $D(n,k,x)$ has multiplicity $N(n,k)$,
where $D(n,k,x)$ is the determinant of the payoff matrix $A(n,k,x)$.
Since $A(n,k)$ is diagonalizable,
the conjecture can be restated as follows: for even $n$,
the  multiplicity of the zero $x=0$ of the determinant of $A(n,k,x)$ equals the
multiplicity of the $0$-eigenvalue of $A(n,k,0)$.   Using Mathematica, we have
verified this conjecture for all even $n$ with $n \le 150$.

The conjecture is easily proved for example when $k=1$, but it is open for
general $k < (n-2)/2$.  In Theorem  9.2, we
prove the conjecture for all $k \ge (n-2)/2$.  
Let  $\nu: = n/2 \in \mathbb{Z}$ and assume throughout
the remainder of this section that $2\nu -1 \ge k \ge \nu -1 \ge  1$.

Our matrix $A(n, k, x)$ is the same as the matrix $M(n, n -k - 1,x)$
defined in \cite{EW2}.  Thus, by
\cite[Lemma 2.3]{EW2}, we have
\begin{equation}\label{eq:9.1}
D(2\nu,k,x)= F_{2\nu-k-1}^2(x),
\end{equation}
where for integers $a \ge 0$,
\begin{equation}\label{eq:9.2}
F_0(x) = 1, \ F_1(x)=x, 
\quad F_{a+2}(x) = (x+1)F_{a+1}(x) -F_{a}(x).
\end{equation}
Induction on $a$ then shows that
\begin{equation}\label{eq:9.3}
F_{3a}(0)=(-1)^a, \quad F_{3a+1}(0) = 0, \quad F_{3a+2}(0)=(-1)^{a+1},
\end{equation}
and for the derivatives,
\begin{equation}\label{eq:9.4}
F'_{3a}(0)=(-1)^a a, \ F'_{3a+1}(0) = (-1)^a (2a+1), 
\ F'_{3a+2}(0)=(-1)^{a} (a+1).
\end{equation}

\begin{lem}
Suppose that  $2\nu-1 \ge k \ge \nu-1 \ge 1$. Then $N(2\nu,k)=2$
if $k \equiv 2\nu+1 \pmod 3$,
and $N(2\nu,k)=0$ otherwise.
\end{lem}

\begin{proof}
By \eqref{eq:9.1} and \eqref{eq:9.3},  $A(2\nu,k)$ is singular if and only if
$k \equiv 2\nu +1 \pmod 3$.
Suppose that this congruence holds.   Since
$A(2\nu,k)$ is a singular skew-symmetric matrix with even rank,
its nullity must be at least $2$.  We must prove that the nullity is
exactly $2$.  Appealing to \cite[Lemma 2.1]{EW2}, we see that
$A(2\nu-1,k)$ has nullity $1$ when $k < 2\nu -1$, and
$A(2\nu+1,k)$ has nullity $1$ when $k = 2\nu -1$.
Thus by Lemma 5.7, the nullity of $A(2\nu,k)$ 
cannot be higher than $2$, so it is exactly $2$.
\end{proof}

\begin{thm}
Suppose that $2\nu -1 \ge k \ge \nu -1 \ge 1$.
Then the zero $x=0$ of $D(2\nu,k,x)$ has multiplicity $N(2\nu,k)$.
\end{thm}

\begin{proof}
By \eqref{eq:9.1} and \eqref{eq:9.3}, $D(2\nu,k,0)$ is nonzero
if and only if $k \not \equiv 2\nu+1 \pmod 3$.
First suppose that $k \not \equiv 2\nu+1 \pmod 3$.
Then the multiplicity of the zero $x=0$ in $D(2\nu,k,x)$ is zero.
By Lemma 9.1,  $N(2\nu,k)=0$.

Next suppose that $k \equiv 2\nu+1 \pmod 3$, so that
the multiplicity of the zero $x=0$ in $D(2\nu,k,x)$ is at least $1$.
By Lemma 9.1, 
$N(2\nu,k) =2$, so in view of  \eqref{eq:9.1}, it remains to show that
the coefficient of $x$ in the polynomial $F_{2\nu-k-1}(x)$ is nonzero.
This follows from the formula for $F'_{3a+1}(0)$ in \eqref{eq:9.4}.
\end{proof}

\end{document}